\documentclass[12pt, draft]{amsart}
\usepackage{amsmath,amssymb}

\allowdisplaybreaks[1]

\title{Riesz means of the Dedekind function II}

\author{Tetsuya Inaba}
\author{Sh\={o}ta  Inoue}

\address{Graduate School of Mathematics,  Nagoya University,
Furocho, 
\newline 
Chikusaku,  Nagoya 464-8602, Japan}
\email{m16005x@math.nagoya-u.ac.jp}

\address{Graduate School of Mathematics,  Nagoya University,
Furocho, 
Chikusaku,  Nagoya 464-8602, Japan}
\email{m16006w@math.nagoya-u.ac.jp}



\theoremstyle{plain}
\newtheorem{theorem}{Theorem}
\newtheorem{lemma}{Lemma}
\newtheorem{corollary}{Corollary}

\theoremstyle{definition}

\numberwithin{theorem}{section}
\numberwithin{lemma}{section}
\numberwithin{corollary}{section}
\numberwithin{remark}{section}
\numberwithin{equation}{section}

\renewcommand{\l}{\left}
\renewcommand{\r}{\right}

\begin{document}

\begin{abstract}
Let  $\psi$ denote  the Dedekind totient function  defined   by 
$
\psi(n)=\sum_{d|n}d\mu^2\l({n}/{d}\r) 
$
with  $\mu$ being  the M\"{o}bius function. 
We shall consider  the    $k$-th Riesz mean of the arithmetical function $n/\psi(n)$ for any non-negative integer $k$
on the assumptions that the Riemann Hypothesis is true, and all the zeros $\rho$ on the critical line of the Riemann zeta function $\zeta$ are simple. 
Our result is an explicit representation of the error term in the formula obtained in a previous work of the second author and I. Kiuchi \cite{IK}.   
We also give an improvement on the error estimate under the assumption of the Gonek-Hejhal Hypothesis. And, we propose a proposition that is equivalent to the Riemann Hypothesis.
\end{abstract}

\maketitle


                                  \section{Statement of Results}


\  Let $s=\sigma+it$  be a complex variable,  where  $\sigma$ and $t$  are real. Let $\varepsilon$ denote an arbitrarily small positive number, 
not necessarily the same ones at each occurrence.
Let $\psi$ be the Dedekind totient function 
defined by 
$
\psi(n)=\sum_{d|n}d\mu^2\l({n}/{d}\r),  
$
where  $\mu$  is   the  M\"{o}bius  function, and let $S_k(x)$ be the $k$-th Riesz mean of $n/\psi(n)$ defined by
\begin{align*}
S_{k}(x)
&=\frac{1}{k!}\sum_{l \leq x}\frac{l}{\psi(l)}\l(1-\frac{l}{x}\r)^{k}.
\end{align*}  
For  any positive real  number  $x (\geq x_{0})$ with  $x_{0}$ being  a sufficiently large positive number,   
the second author and I. Kiuchi \cite{IK} showed that the asymptotic relation 
\begin{align}                                                                                    \label{SITA}                     
S_k(x)=                                            
\frac{\zeta(4)h(1)}{(k+1)!\zeta(2)}x  + E_{k}(x)     
\end{align}
holds  with  
\begin{align}
h(s):=\prod_p\l( 1+\frac{1}{p^{s+2}(1+1/p)}-\frac{1}{p^{2s+s}(1+1/p)} \r),
\end{align}
where $E_{k}(x)$  is  the error term. 
It was proved in \cite{IK} that     
\begin{align}                                        \label{SSS}
E_{k}(x)  \ll \frac{x^{-\frac12 +\varepsilon}}{k}    
\end{align}
under the assumption that the Riemann Hypothesis is true.

Before the statement of our theorem, we define  the function  $h_{n}(s)$  by  
\begin{align}                                        \label{hns}
&h_{n}(s) := \prod_{p}\Bigg(1 + 
\sum_{m=1}^{2^{n}-1}\frac{1}{p^{ms+m+1} \l(1 + \frac{1}{p}\r)} - \frac{1}{p^{2^n s + 2^n}\l(1 + \frac{1}{p} \r)} \Bigg),
\end{align}
which  is   absolutely and uniformly convergent in any compact set in the half-plane 
${\rm Re}~s  > -1 + \frac{1}{2^n}$ (see Lemma \ref{lem1}).     
The  purpose of this paper is to obtain an explicit representation of $E_k(x)$ of the general $k$-th 
Riesz mean under the assumption that the Riemann Hypothesis is true, 
and all zeros of the Riemann zeta-function $\zeta(s)$ on the critical line is simple.


\begin{theorem} \label{th1}
Suppose that the Riemann Hypothesis is true, and 
all the  zeros $\rho$ on the critical   line  of the Riemann zeta-function $\zeta(s)$ are  simple. 
Then,  for any positive integers  $k \geq 2$  and   $n > \frac{\log{4\log{x}}}{2\log 2}$, 
there exists a number $T$ \ $(x^4 \leq T \leq x^4 + 1)$  such that  
\begin{align}                                                                                       \label{RPSI}                                   
E_k(x)&= Y_{k,n}(x,T)  x^{- \frac{1}{2}}  + O\l( x^{-1+\frac{C}{\sqrt{\log{x}}}}\left( \frac{\sqrt{\log{x}}}{(k-1)!} + \frac{1}{k} \right) \r)     
\end{align}
with an absolute constant $C>0$,  where 
\begin{align}                                                                                            \label{YYY}
&{Y}_{k,n}(x, T) \\
&:={\rm Re} \sum_{0< \gamma <T}\frac{\zeta\l( -\frac12+i\gamma \r)\zeta(2^{n-1}+2^ni\gamma)h_n\l(-\frac12+i\gamma\r)}{\zeta'\l( \frac{1}{2} + i\gamma \r)}    \nonumber \\ 
& \qquad  \times  \frac{x^{i\gamma}}{\l( -\frac{1}{2} + i\gamma \r)\l(\frac{1}{2} + i\gamma\r) 
  \cdots \l(k - \frac{1}{2} + i\gamma\r)}.  \nonumber 
\end{align}  
\end{theorem}

It may be interesting to study whether we can remove the $\varepsilon$-factor from the right-hand side of \eqref{SSS} or not.
The second author and Kiuchi made use of   
the  Gonek-Hejhal Hypothesis   
(S. M. Gonek \cite{G} and D.  Hejhal \cite{H} independently conjectured),  namely   
\begin{align}
J_{-\lambda}(T)& := \sum_{0 < \gamma \leq T}\frac{1}{|\zeta'(\rho)|^{2\lambda}} 
                  \asymp T(\log{T})^{(\lambda-1)^2}                                                                   \label{GH}
\end{align} 
for real number  $\lambda < \frac{3}{2}$ to improve the estimate \eqref{SSS} of the error  term  $E_{k}(x)$.
In the present paper we have the following Theorem.


\begin{theorem}   \label{th2}
If  the  function   $J_{-1/2}(T)$ is estimated by  $O\l(T^{1+\varepsilon}\r)$ 
for any small fixed positive number $\varepsilon$, then 
the estimate 
$
E_{k}(x) = O\l(x^{-\frac{1}{2}}\r)
$
holds for any positive integer $ k \geq 2$. 
\end{theorem}

Define    
$$
F(s):=\sum_{n=1}^{\infty}\frac{n}{\psi(n)n^s}
$$
for ${\rm Re}~s>1$.
The above results will be proved by considering analytic continuation of $F(s)$. 
Also, the second author and Kiuchi studied the case of the Euler totient function by the same method in \cite{IK2}.
We also obtain the following results by analytic continuation of $F(s)$.

\begin{theorem}   \label{th3}
Let $k \geq 0$, and $\Theta$ is supremum of real parts of non-trivial zeros of $\zeta(s)$. Then we have 
\begin{align}
E_k(x)=\Omega_\pm(x^{\Theta-1-\varepsilon}).
\end{align}
\end{theorem}

\begin{corollary}  \label{th4}
The Riemann Hypothesis is equivalent to the statement that there exists an integer $k\geq 2$ such that $E_k(x)\ll x^{-\frac{1}{2}+\varepsilon}$.
\end{corollary}


                                                   \section{Some Lemmas} 


In order to prove Theorems \ref{th1} and  \ref{th2},   we shall prepare the following lemmas.


\begin{lemma}   \label{lem1}
For any  positive  integer $n$  and   ${\rm Re}~s>1$,  we have 
\begin{align}                                                                                                            \label{F}  
F(s) = \frac{\zeta(s)\zeta(2^n s + 2^n)}{\zeta(s + 1)} h_n(s)
\end{align}    
where $h_n(s)$ is defined by (\ref{hns}), and is absolutely and uniformly convergent in any compact set in the half-plane ${\rm Re}~s  > -1 + \frac{1}{2^n}$.
\end{lemma}

\begin{proof}
 We use induction on $n$. The lemma is true for $n = 1$, which is  Lemma 2.1 in  \cite{IK}.
Now,  we assume that  the statement  (\ref{F}) holds  for  $2, 3,  \ldots, n-1$.  
The induction assumption  tells us 
\begin{align*}
F(s) 
&= \frac{\zeta(s)\zeta(2^{n - 1} s + 2^{n - 1})}{\zeta(s + 1)} h_{n - 1}(s)\\
&= \frac{\zeta(s)\zeta(2^{n}s + 2^n)}{\zeta(s+1)}\prod_{p}\l( 1 + \frac{1}{p^{2^{n-1}s + 2^{n-1}}} \r)h_{n - 1}(s).
\end{align*}
Now, for ${\rm Re}~s > 1$, we have
\begin{align*}
&\prod_{p}\l( 1 + \frac{1}{p^{2^{n-1}s + 2^{n-1}}} \r)h_{n - 1}(s)\\
&= \prod_{p}\l(1 + \frac{1}{p^{2^{n - 1} s + 2^{n - 1}}}\r)   
\l(1 + \sum_{m=1}^{2^{n-1}-1}\frac{1}{p^{ms+m+1}\l(1 + \frac{1}{p}\r)}
        - \frac{1}{p^{2^{n - 1} s + 2^{n - 1}}\l(1 + \frac{1}{p}\r)}\r)   \\
&= \prod_{p} \bigg(1 + \sum_{m=1}^{2^{n-1}-1}\frac{1}{p^{ms + m+1}\l(1 + \frac{1}{p}\r)} 
 +  \frac{1}{p^{2^{n-1}s+2^{n-1}+1}\l(1+\frac{1}{p}\r)}  \\  
& \qquad  
+ \sum_{m=1}^{2^{n-1}-1}\frac{1}{p^{(m+2^{n-1})s+(m+2^{n-1})+1} \l(1 + \frac{1}{p} \r)} 
- \frac{1}{p^{2^{n} s + 2^{n}}\l(1 + \frac{1}{p} \r)}\bigg)  \\
&= \prod_{p}\bigg(1 + 
\sum_{m=1}^{2^{n}-1}\frac{1}{p^{ms+m+1}\l(1 + \frac{1}{p}\r)} -
\frac{1}{p^{2^n s + 2^n}\l(1 + \frac{1}{p} \r)}  \bigg) = h_n(s),
\end{align*}
which completes the proof of Lemma \ref{lem1}.
\end{proof}


\begin{lemma}   \label{lem11}
Let $x$ be any sufficiently large real number,  and  let $\delta=1/\sqrt{\log x}$. 
For any  positive  integer $n (> \frac{\log({2}/{\delta})}{\log 2} >8)$,  we have  
\begin{align}                                                                                                         \label{estimate_h_n}
h_n(-1+\delta+it)  \ll \exp\left( \frac{C}{\delta} \right)
\end{align}
with an absolute constant $C>0$. 
\end{lemma}

\begin{proof}
We use  (\ref{hns}) and  the inequality   
$
\pi(u) \leq C_1\frac{u}{\log u}
$  for any positive number $u \geq 2$  
to obtain an upper bound for  
the function $h_n(-1+\delta+it)$ for any  positive  integer $n (> \frac{\log({2}/{\delta})}{\log 2} >8)$, 
namely  
\begin{align*}
|h_{n}&(-1+\delta+it)|  
= \prod_p\l( 1+\sum_{m=1}^{2^n-1}\frac{1}{p^{m\delta+1}\l( 1+\frac{1}{p} \r)}+\frac{1}{p^{2^n\delta}\l( 1+\frac{1}{p} \r)} \r) \\
&\leq \prod_p\l( 1+\sum_{m=1}^{\infty}\frac{1}{p^{m\delta}(p+1)}+\frac{1}{p(p+1)} \r) \\
&= \prod_p \l( 1+\frac{1}{(p+1)(p^\delta-1)}+\frac{1}{p(p+1)} \r) \\
&=\exp\l({\sum_p\log{\l( 1+\frac{1}{(p+1)(p^\delta-1)}+\frac{1}{p(p+1)} \r)}}\r) \\
&\leq\exp\l( \sum_p\l( \frac{1}{(p+1)(p^\delta-1)}+\frac{1}{p(p+1)} \r) \r) 
\ll\exp\l( \sum_p\frac{1}{(p+1)(p^\delta-1)} \r) \\
&=\exp{\lim_{\tau\rightarrow \infty}\l( \frac{\pi(\tau)}{(\tau-1)(\tau^\delta-1)}+\int_2^{\tau}\l( \frac{\pi(u)}{(u-1)^2(u^\delta-1)}+\frac{\pi(u)\delta u^\delta}{(u-1)^2(u^\delta)^2} \r)du \r) }  \\
&\leq\exp{\l( C_1\int_2^\infty \frac{du}{\frac{(u-1)^2}{u}(u^\delta-1)\log{u}}+C_1\delta\int_2^\infty\frac{du}{\frac{(u-1)^2(u^\delta-1)^2}{u^{1+\delta}}\log{u}} \r)} \\
&=\exp{\l( \frac{2}{\delta}C_1\int_2^\infty \frac{du}{\frac{(u-1)^2}{u}(\log{u})^2}+\frac{8}{\delta}C_1\int_2^\infty\frac{du}{\frac{(u-1)^2}{u}(\log{u})^3} \r)} \\
&\ll \exp\left( \frac{C}{\delta} \right)
\end{align*}
with an absolute constant $C > 0$.  This completes the proof of Lemma \ref{lem11}. 
\end{proof}


\begin{lemma} \label{lem2}
Assume that  the Riemann hypothesis is true. Then,  there exists a number $t \in [T, T+1]$ such that 
\[                                                                                                     
\zeta(\sigma+it)             \ll t^{\varepsilon} \qquad  {\rm and} \qquad  
\frac{1}{\zeta(\sigma+it)}   \ll t^{\varepsilon} 
\]
for  every $1/2\leq \sigma \leq 2$ and any sufficiently large real number $T > 0$.
\end{lemma}

\begin{proof} \ 
The first assertion is given by (14.2.5) and (14.14.1), and the second assertion is given by (14.16.2) in the textbook \cite{T},
respectively.
\end{proof}


                                          \section{Proof  of Theorem \ref{th1}}


\begin{proof}
Suppose that the Riemann Hypothesis is true, and 
all the  zeros $\rho$ on the critical   line  of the Riemann zeta-function $\zeta(s)$ are  simple. 
Let $x$ be any sufficiently large real number,   and  let  $\delta = 1 / \sqrt{\log{x}}$.
Let $T \in [x^{4}, x^4 + 1]$ satisfy the condition Lemma \ref{lem2} and
$n$ be a positive integer satisfying  $n~(> \frac{\log({2}/{\delta})}{\log 2} )$.
We make use of Lemma \ref{lem1} and (5.19) in H. Montgomery and R. C. Vaughan \cite{MV} with $\sigma_0 := 1 + \frac{1}{\log{x}}$ to obtain 
\begin{align}                                                                                                      
S_{k}(x)
&= \frac{1}{2\pi i}\int_{\sigma_0 - i\infty}^{\sigma_0 + i\infty}F(s)\frac{x^s}{s(s+1)(s+2)\cdots (s+k)}ds 	\label{S1}	\\
&=\frac{1}{2\pi i}\int_{\sigma_0 - iT}^{\sigma_0 + iT}F(s)\frac{x^s}{s(s+1)(s+2)\cdots (s+k)}ds 
			+ O\l({x}T^{-k+\varepsilon}\r).                                                                        \label{S}
\end{align}
Now, we move the line of integration to ${\rm Re}~s = -1 + \delta$.  
In the rectangular contour  formed  by  the line segments  joining  the points 
$\sigma_0-iT$, $\sigma_0+iT$, $-1+\delta+iT$, $-1+\delta-iT$  and  $\sigma_0-iT$ counter-clockwise,
we observe that  $s=1$ is a simple pole, and   
                 $s=-\frac{1}{2}+i \gamma$ is also a simple pole of the integrand. 
Thus we get the main term from the sum of the residues coming from the poles  $s=1$ and $s=-\frac{1}{2}+i\gamma$. 
That is,  
\begin{align}                                                                                                      
&\frac{1}{2\pi i}\int_{\sigma_0-iT}^{\sigma_0+iT}F(s)\frac{x^s}{s(s+1)(s+2)\cdots (s+k)}ds                        \label{S_T1}  \\
&=\frac{1}{2\pi i}\left\{\int_{-1+\delta+iT}^{\sigma_0+iT}
+ \int_{-1+\delta-iT}^{-1+\delta+iT} +  \int_{\sigma_0-iT}^{-1+\delta-iT}\right\}
F(s)\frac{x^s}{s(s+1)\cdots (s+k)}ds          \nonumber                                                  \\
&  + \sum_{0<|\gamma|<T}\underset{s=-\frac{1}{2}+i\gamma}{\rm Res}\l(F(s)\frac{x^s}{s(s+1)\cdots(s+k)}\r)
+ \underset{s=1}{\rm Res}\l(F(s)\frac{x^s}{s(s+1)\cdots(s+k)}\r).        \nonumber         
\end{align}
The last term on the  right-hand side of (\ref{S_T1}) are evaluated by the second author and Kiuchi  \cite{IK},
who  have  calculated  that 
\begin{align*}
\underset{s=1}{\rm Res}\l(\frac{F(s)x^s}{s(s+1)\cdots(s+k)}\r)
= \frac{\zeta(4)}{(k+1)!\zeta(2)}\prod_p{\l( 1+\frac{1}{p^2(p+1)}-\frac{1}{p^3(p+1)} \r)x}.                            
\end{align*}

Furthermore, we  have     
\begin{align*}
&\sum_{0<|\gamma|<T}\underset{s=-\frac{1}{2}+i\gamma}{\rm Res}\l(F(s)\frac{x^s}{s(s+1)\cdots(s+k)}\r) \\ 
&= \sum_{0 < \gamma < T}{\rm Re}~\bigg(\frac{\zeta\l(  -\frac{1}{2}+ i\gamma \r)\zeta\l( 2^{n-1}+2^n i\gamma \r)}{\zeta'\l( \frac{1}{2} + i\gamma \r)}h_n\l(- \frac{1}{2} + i\gamma \r)\\
& \qquad \times \frac{x^{i\gamma}}{\l( -\frac{1}{2} + i\gamma \r)(\frac{1}{2} + i\gamma)
 (\frac{3}{2} + i\gamma) \cdots (k - \frac{1}{2} + i\gamma)}\bigg)x^{- \frac12}
\end{align*}  
by the assumptions.\\

Let  $T\geq T_0$,  where $T_0$ is a sufficiently large real number.  
Denote by $Q_k(x)$ the second term  (the left vertical line segment)  of the integral  
on the right-hand side  of (\ref{S_T1}). Using (\ref{estimate_h_n}), we have
\begin{align}                                                                                             \label{Q}
&Q_{k}(x) 
:= \frac{1}{2\pi} \int_{-T}^{T}
\frac{F\l(-1+\delta+it\r)x^{-1+\delta+it}}{(-1+\delta+it)(\delta+it)\cdots (k-1+\delta+it)}dt  \\
=&\frac{x^{-1+\delta}}{2\pi} \l(\int_{|t|\leq T_0}+\int_{T_{0}\leq |t| \leq T}\r)\frac{\zeta(-1+\delta+it)\zeta(2^n\delta+2^n it)}{\zeta(\delta+it)}
\times   \nonumber \\ 
& \quad  \frac{ h_n(-1+\delta+it)x^{it}}
{(-1+\delta+it)(\delta+it)\cdots (k-1+\delta+it)}dt  \nonumber \\
\ll& \frac{x^{-1+\delta}}{\delta(k-1)!}\exp\l( \frac{C}{\delta} \r)  \nonumber  \\
&+    x^{-1+\delta}\exp\l( \frac{C}{\delta} \r)\int_{T_{0}\leq |t| \leq T} 
\l|\frac{t^{\frac{3}{2}-\delta}\zeta(2-\delta - it)}
{t^{\frac{1}{2}-\delta}\zeta(1-\delta-it) t^{k+1}}\r| dt,  \nonumber 
\end{align}
where $C > 0$ is an absolute constant.
Using  Lemma \ref{lem2} we have 
\begin{align}                                                                                                     
	Q_{k}(x) &\ll x^{-1+\frac{C}{\sqrt{\log{x}}}}\left( \frac{\sqrt{\log{x}}}{(k-1)!} + \frac{1}{k} \right)                    \label{Q1k}
\end{align}
for any positive integer $k\geq 2$. 

Next we estimate  the contributions  coming  from  the upper  horizontal line (the lower horizontal line is similar). 
First we split the integral as
\begin{align*}
&\int_{-1+\delta+iT}^{\sigma_0+iT}F(s)\frac{x^s}{s(s+1)\cdots(s+k)}ds\\
&=\l(\int_{\frac{1}{2}+iT}^{\sigma_0+iT}+\int_{-\frac{1}{2}+iT}^{\frac{1}{2}+iT}+\int_{-1+\delta+iT}^{-\frac{1}{2}+iT} \r)F(s)\frac{x^s}{s(s+1)\cdots(s+k)}ds\\
&=:I_1 + I_2 + I_3, 
\end{align*}
say. We use the functional equation for the Riemann zeta-function, Lemmas \ref{lem1} and \ref{lem2} to obtain
\begin{align*}
	|I_1| 	
	&\leq \int_{\frac{1}{2}}^{\sigma_0}|\frac{\zeta(\sigma+iT)\zeta(2^n(\sigma+iT)+2^n)h_n(\sigma+iT)x^\sigma}{\zeta(1+\sigma+iT)T^{k+1}} |d\sigma\\
  &\ll\int_{\frac12}^{\sigma_0}\frac{T^\varepsilon x}{T^{k+1}}d\sigma
  \ll xT^{-k+\varepsilon}.
\end{align*}
Similarly, we have
\begin{align*}
  |I_2|
  &\ll \int_{-\frac12}^\frac12|T^{\frac12-\sigma}\zeta(1-\sigma-iT)T^\varepsilon h_n(\sigma+iT)\frac{x^\sigma}{T^{k+1}}|d\sigma\\
  &\ll \int_{-\frac12}^\frac12 T^{2\varepsilon}\frac{x^\frac12}{T^{k+1}}d\sigma
  \ll x^\frac12 T^{-k+2\varepsilon},
\end{align*}
and 
\begin{align*}
  |I_3|		
  &\ll \int_{-1+\delta}^{-\frac12}|\frac{T^{\frac12-\sigma}\zeta(1-\sigma-iT)x^\sigma}{T^{-\frac12-\sigma}\zeta(-\sigma-iT)T^{k+1}}|d\sigma \\
  &\ll \int_{-1+\delta}^{-\frac12}\frac{T^{\frac32-\delta}T^\varepsilon x^{-\frac12}}{T^{k+1}}
  \ll x^{-\frac12}T^{\frac12-\delta+\varepsilon-k}
\end{align*}
for a positive number  $T \ (x^{4}\leq T\leq x^{4}+1)$.

Hence using  horizontal lines defined by  ${\rm Re}s=\pm T$ to move the line of integration in (\ref{S_T1}),
we find that the total contribution of the horizontal lines in absolute value is  
\begin{align}                                                                                           \label{Q12}
&\ll xT^{-k+\varepsilon}.
\end{align}
Applying the relation (\ref{S_T1})  and    the error estimates   (\ref{Q1k}),   (\ref{Q12})  to  (\ref{S}), 
we obtain,  for $k\geq 2$,  
\begin{align}                                                                                                     
S_{k}(x)				\label{EEE}
=  \frac{\zeta(4)h(1)}{(k+1)!\zeta(2)}x &	+ {Y}_{k,n}(x, T)x^{- \frac{1}{2}}\\    \nonumber 
&+ O\l( x^{-1+\frac{C}{\sqrt{\log{x}}}}\left( \frac{\sqrt{\log{x}}}{(k-1)!} + \frac{1}{k} \right) \r) 
\end{align}
with  a positive number  $T\ (x^{4}\leq T \leq x^{4}+1)$, 
where  
\begin{align}                                                                                            \label{YYY}
&{Y}_{k,n}(x, T) \\
&:={\rm Re} \sum_{0< \gamma <T}\frac{\zeta\l( -\frac12+i\gamma \r)\zeta(2^{n-1}+2^ni\gamma)h_n\l(-\frac12+i\gamma\r)}{\zeta'\l( \frac{1}{2} + i\gamma \r)}    \nonumber \\ 
& \qquad  \times  \frac{x^{i\gamma}}{\l( -\frac{1}{2} + i\gamma \r)\l(\frac{1}{2} + i\gamma\r) 
  \cdots \l(k - \frac{1}{2} + i\gamma\r)},   \nonumber 
\end{align}  
which completes the proof of    
the identity  (\ref{RPSI}).  
\end{proof}


\section{Proof of Theorem  \ref{th2}}


We use  (\ref{YYY}), Lemma \ref{lem2}, the functional equation for the Riemann zeta-function and Stirling's formula for $\chi(s)$ 
to  obtain  
\begin{align*}
&\sum_{0<\gamma < T}{\rm Re}\bigg(
 \frac{\chi\l( -\frac12+i\gamma \r)\zeta\l( \frac32+i\gamma \r)\zeta(2^{n-1}+2^ni\gamma)h_n\l(-\frac12+i\gamma\r)}{\zeta'\l( \frac{1}{2} + i\gamma \r)}    \\
& \qquad  \times  \frac{x^{i\gamma}}{\l( -\frac{1}{2} + i\gamma \r)\l(\frac{1}{2} + i\gamma\r) 
  \cdots \l(k - \frac{1}{2} + i\gamma\r)} \bigg)  \\
&\ll \sum_{0<\gamma < T}
\frac{1}{\gamma^{k-\frac{1}{2}- \varepsilon} |\zeta'(\frac{1}{2}+i\gamma)|}.  
\end{align*}
It suffices to show that ${Y}_{k,n}(x, T)$ converges for $k=2$.  
Using  (\ref{GH})   and partial summation   
we obtain   
\begin{align*}
&\sum_{0<\gamma < T}
\frac{1}{\gamma^{\frac{3}{2}- \varepsilon} |\zeta'(\frac{1}{2}+i\gamma)|}  
\ll \left[\frac{J_{-\frac12}(t)}{t^{\frac32 -\varepsilon}}\right]_{14}^{T}
+ \int_{14}^{T}\frac{J_{-\frac12}(t)}{t^{\frac52 -\varepsilon}}dt \ll 1,
\end{align*}
which implies the estimate $E_{k}(x)=O\l(x^{-\frac12}\r)$ for any positive integer $k\geq 2$.
\qed


\section{Proof of Theorem  \ref{th3} and Corollary \ref{th4}}

\begin{proof}[Proof of Theorem \ref{th3}]
We use reductio ad absurdum. Namely, we suppose that there exists a number $\frac{1}{5}>\varepsilon_0>0$ such that 
\begin{align*}
S_k(x)-\frac{\zeta(4)h(1)}{(k+1)!\zeta(2)}x\leq x^{\Theta-1-\varepsilon_0} \hspace{15pt} (x>K(\varepsilon_0)).
\end{align*}
Define $G(s)$ by 
\begin{align*}
G(s):=\int_1^\infty \frac{S_k(x)-\frac{\zeta(4)h(1)}{(k+1)!\zeta(2)}x-x^{\Theta-1-\varepsilon_0}}{x^{s+1}}dx
\end{align*}
for $\sigma>1$.
We use the Mellin transform, and we have
\begin{align*}
G(s)=\frac{\zeta(s)\zeta(4s+4)}{\zeta(s+1)s\cdots(s+k)}h_2(s)-\frac{\zeta(4)h(1)}{(k+1)!\zeta(2)}\frac{1}{s-1}-\frac{1}{s+1-\Theta+\varepsilon_0}
\end{align*}
by $(\ref{S1})$. So, we have
\begin{align*}
&\int_1^\infty  \frac{S_k(x)-\frac{\zeta(4)h(1)}{(k+1)!\zeta(2)}x-x^{\Theta-1-\varepsilon_0}}{x^{s+1}}  dx \\
&=\frac{\zeta(s)\zeta(4s+4)}{\zeta(s+1)s\cdots(s+k)}h_2(s)-\frac{\zeta(4)h(1)}{(k+1)!\zeta(2)}\frac{1}{s-1}-\frac{1}{s+1-\Theta+\varepsilon_0} \\
\end{align*}
for $\sigma>1$. 
Notice that the integral term on the left-hand side is absolutely convergent for $\sigma>-1+\Theta-\varepsilon_0$ 
since the right-hand side is a definite value when $s=\sigma>-1+\Theta-\varepsilon_0$. However, there is a zero of $\zeta(s+1)$  from the definition of $\Theta$, and so
\begin{align*}
\frac{\zeta(s)\zeta(2s+2)}{\zeta(s+1)s\cdots(s+k)}h(s)-\frac{\zeta(4)h(1)}{(k+1)!\zeta(2)}\frac{1}{s-1}-\frac{1}{s+1-\Theta+\varepsilon_0}
\end{align*}
is not regular for $\sigma>-1+\Theta-\varepsilon_0$, which is a contradiction. Therefore 
\begin{align*}
E_k(x)=S_k(x)-\frac{\zeta(4)h(1)}{(k+1)!\zeta(2)}x=\Omega_+(x^{\Theta-1-\varepsilon})
\end{align*}
holds for any positive $\varepsilon$.

Similarly, we have
\begin{align*}
E_k(x)=S_k(x)-\frac{\zeta(4)h(1)}{(k+1)!\zeta(2)}x=\Omega_-(x^{\Theta-1-\varepsilon}).
\end{align*}
\end{proof}

\begin{proof}[Proof of Corollary \ref{th4}]
The Riemann Hypothesis is sufficient for the statement that there exists 
an integer $k\geq 2$ such that $E_k(x)\ll x^{-\frac{1}{2}+\varepsilon}$ by \eqref{SSS}. 
We also see that the inverse assertion holds by Theorem \ref{th3}.
\end{proof}

%






\begin{thebibliography}{99}

\bibitem{G} S. M. Gonek, On negative moments of the Riemann zeta-function, {\it Mathematika} {\bf 36} (1989), 71--88

\bibitem{H} D. Hejhal,  On the distribution of $\log{\left|\zeta'(\frac{1}{2}+it)\right|}$, {\it Number theory, 
trace formula and discrete groups}  (ed. K. E. Aubert, E. Bombieri and D. Goldfeld, Academic Press, San Diego, 1989) 343--370.


\bibitem{IK} S. Inoue and I. Kiuchi,    
\ Riesz mean of the Dedekind function,    
{\it Kyushu J.\ Math.}   {\bf 71}  (2017),  105--114.


\bibitem{IK2} S. Inoue and I. Kiuchi, Riesz means of the Euler totient function, {\it Funct.\ Approx.\ Comment.\ Math}, to appear.


\bibitem{MV} H. L. Montgomery and R. C. Vaughan, {\it Multiplicative Number Theory I. Classical Theory}, 
Cambridge Studies in Advanced Mathematics, Cambridge University press, 2007.


\bibitem{SS1} A. Sankaranarayanan  and  S. K. Singh,    
\ On the Riesz means of $\frac{n}{\phi(n)}$,  
{\it Hardy--Ramanujan J.} {\bf 36}  (2013),  8--20.      


\bibitem{SS2} A. Sankaranarayanan  and  S. K. Singh,    
\ On the Riesz means of $\frac{n}{\phi(n)}$ - II,  
{\it Arch.\ Math.}  (Basel) {\bf 103}  (2014),  329--343.      


\bibitem{SS3} A. Sankaranarayanan  and  S. K. Singh,    
\ On the Riesz means of $\frac{n}{\phi(n)}$ -  III,  
{\it Acta Arith.}   {\bf 170}  (2015),  275--286.      


\bibitem{SI} R. Sitaramachandrarao,    
\ On an error term of Landau - II,    
{\it Rocky Mountain J.\ Math.}    {\bf 15}  (1985),  579--588.      


\bibitem{T}  E. C. Titchmarsh,\ {\it The Theory of the Riemann Zeta-Function}, 
Second Edition, Edited and with a preface by D. R. Heath--Brown, 
The Clarendon Press, Oxford University Press, New York, 1986.  



\end{thebibliography}
\end{document}